\documentclass[12pt]{amsart}
\usepackage{amssymb,amsmath,amsthm,latexsym}
\usepackage{mathrsfs}
\usepackage{a4wide}
\usepackage{collectbox}
\usepackage[usenames,dvipsnames]{color}
\usepackage[utf8]{inputenc}
\usepackage[T1]{fontenc}
\usepackage{dsfont}
\usepackage{tikz}
\usepackage{float}
\usepackage[mathscr]{euscript}
\usepackage{mathtools}
\usepackage{hyperref}
\usepackage{verbatim}

\newtheorem*{theorem}{Main Theorem}
\newtheorem*{corollary}{Corollary}

\newtheorem{lemma}{Lemma}
\theoremstyle{remark}

\newtheorem*{remark*}{Remark}
\newtheorem{remark}{Remark}
\theoremstyle{definition}

\numberwithin{equation}{section}
\makeatother

\newcommand{\vertiii}[1]{{\left\vert\kern-0.25ex\left\vert\kern-0.25ex\left\vert #1 
		\right\vert\kern-0.25ex\right\vert\kern-0.25ex\right\vert}}

\newcounter{smallromans}

\newenvironment{romanenumerate}
{\begin{list}{{\normalfont\textrm{(\roman{smallromans})}}}%
		{\usecounter{smallromans}\setlength{\itemindent}{0cm}%
			\setlength{\leftmargin}{5.5ex}\setlength{\labelwidth}{5.5ex}%
			\setlength{\topsep}{.5ex}\setlength{\partopsep}{.5ex}%
			\setlength{\itemsep}{0.1ex}}}%
	{\end{list}}

\newcounter{smallromansdash}

	{\end{list}}

\newcounter{bigromans} 
	{\end{list}}

\title[Kalton's constant is not smaller than 3]{Approximate modularity: Kalton's constant\\ is not smaller than 3}
\author{Micha{\l} Gnacik}
\address{School of Mathematics and Physics, Lion Gate Building, Lion Terrace, University of Portsmouth, Portsmouth, United Kingdom}
\email{michal.gnacik@port.ac.uk}
\author{Marcin Guzik}
\address{UBS Business Solutions, Krakowska 280, 32-080 Zabierz\'{o}w, Poland}
\email{marcin.guzik@ubs.com}
\author{Tomasz Kania}
\address{Institute of Mathematics, Czech Academy of Sciences, \v{Z}itn\'{a} 25, 115~67 Prague 1, Czech Republic and Institute of Mathematics, Jagiellonian University, {\L}ojasiewicza 6, 30-348 Krak\'{o}w, Poland}
\email{kania@math.cas.cz, tomasz.marcin.kania@gmail.com}
\date{\today}

\begin{document}
	\maketitle
	\begin{abstract}
		Kalton and Roberts [\emph{Trans.~Amer.~Math.~Soc.}, 278 (1983), 803--816] proved that there exists a universal constant $K\leqslant 44.5$ such that for every set algebra $\mathcal{F}$ and every 1-additive function $f\colon \mathcal{F}\to \mathbb R$ there exists a finitely-additive signed measure $\mu$ defined on $\mathcal{F}$ such that $|f(A)-\mu(A)|\leqslant K$ for any $A\in \mathcal{F}$. The only known lower bound for the optimal value of $K$ was found by Pawlik [\emph{Colloq.~Math.}, 54 (1987), 163--164], who proved that this constant is not smaller than $1.5$; we improve this bound to $3$ already on a non-negative 1-additive function.
		
		\noindent    
	\end{abstract}
	
	\section{Introduction}
	
	Let $\Omega$ be a set, $\mathcal{F}$ be a set algebra over $\Omega$, and $\Delta \geqslant 0$. A function $f \colon \mathcal{F} \rightarrow \mathbb{R}$ is $\Delta$-\emph{additive}, whenever $f(\varnothing) = 0$ and 
	$$|f(A)+f(B) - f(A \cup B) | \leqslant \Delta\quad (A,B\in \mathcal{F}, A\cap B=\varnothing).$$
	Quite clearly, 0-additive maps are nothing but signed, finitely-additive measures on $\mathcal{F}$. For brevity, we refer to $0$-additive functions as additive. Kalton and Roberts proved in \cite{KRo} a rather surprising stability theorem for $\Delta$-additive maps, which asserts that there exists a universal constant (we follow Pawlik's convention \cite{Paw} and refer to it as \emph{Kalton's constant}) $K \leqslant 44.5$ (independent of the choice
	of $\mathcal{F})$ such that for every $\Delta$-additive function $f\colon \mathcal{F}\to \mathbb R$ there exists a (signed, finitely-additive) measure $\mu\colon \mathcal{F}\to \mathbb R$ such that
	\begin{equation}\label{eq: kalton_add}\sup_{A\in \mathcal{F}} |f(A) - \mu(A)| \leqslant K\cdot \Delta. \end{equation}
	In 2014, Bondarenko, Prymak, and Radchenko decreased the upper bound for $K$ from 44.5 to 38.8 (see \cite[Proof of Corollary 1.2]{Bondarenko}).\smallskip
	
	Results of this kind (that is, including ours) are of importance in Functional Analysis, for example, in the theory of twisted sums of (quasi-)Banach spaces and certain stability problems of vector measures \cite{KRo, kochanek}. Improving Kalton's constant may likely fine-tune various optimisation algorithms in machine learning and algorithmic game theory (see \cite[Section 1.2]{Feige} and references therein for more details). Moreover, recently there have been efforts to extend the validity of the Kalton--Roberts theorem to lattices \cite{badora}.\medskip
	
	An analogous and closely related stability problem for $\varepsilon$-modular ($\varepsilon >0$) set functions was recently studied by Feige, Felfman, and Talgam-Cohen \cite{Feige}. 
	In order to present it, we require a~piece of notation.\smallskip
	
	Let $m \in \mathbb{N}$ (we are reserving $n$ for a different purpose), assume that $\Omega_m= \{1, 2, \ldots, m\}$  (to fit the setting in  \cite{Feige}), set $\mathcal{F}:=  2^{\Omega_m}$ and let $f\colon\mathcal{F} \rightarrow \mathbb{R}$ be a function such that $f(\varnothing) = 0$, where $2^{\Omega_m}$ denotes the power set of $\Omega_m$. Then, $f$ is additive (that is, it is a~finitely-additive signed measure), if and only if, it satisfies the \emph{modular identity}:
	$$ f(A) + f(B) = f(A \cup B) + f(A \cap B) \quad (A, B \in \mathcal{F}).$$ 
	Functions that assume a~possibly non-zero value at the empty set and that satisfy the modular identity are for this reason called \emph{modular}. For $\varepsilon >0$, a function $f\colon \mathcal{F} \rightarrow \mathbb{R}$ is then termed 
	\emph{$\varepsilon$-modular}, whenever
	\begin{equation} \label{eq: mod}|f(A)+f(B) - f(A \cup B) - f(A \cap B) | \leqslant \varepsilon \quad (A, B \in \mathcal{F}).\end{equation}
	Also, $f$ is said to be  \emph{weakly}-$\varepsilon$-\emph{modular}, whenever (\ref{eq: mod}) is satisfied for every sets $A$, $B$ so that $A \cap B = \varnothing$, in particular, if $f(\varnothing)=0$, then the properties of
	being weakly-$\varepsilon$-modular and $\varepsilon$-additive are equivalent. Moreover, every weakly-$\varepsilon$-modular function is $2\varepsilon$-modular (see \cite[Proposition 2.1]{Feige}). \smallskip
	
	The main results in \cite{Feige} state that
	there are universal constants  $K_s < 12.62$ (\emph{the strong Kalton constant}) and $K_w < 24$ (\emph{the weak Kalton constant}) so that for every $\varepsilon$-modular function $f$ there is a modular function $\nu_1$ with
	\begin{equation}\label{eq: kalton_mod}\sup_{A \in \mathcal{F}}|f(A) - \nu_1(A)| \leqslant \varepsilon K_s,\end{equation}
	and for every weakly-$\varepsilon$-modular function $h$ there is a modular function $\nu_2$ with \begin{equation}
	\label{eq: kalton_wadd}
	\sup_{A \in \mathcal{F}}|h(A) - \nu_2(A)| \leqslant \varepsilon K_w.\end{equation}
	It is also worth emphasising the inequalities between $K_s$ and $K_w$ \cite[Corollary 2.7]{Feige}), namely
	\begin{equation}\label{eqn:kalt_eqn} \tfrac{1}{2}K_w \leqslant K_s \leqslant K_w.\end{equation}
	
	Here, the constants $K_w$ and $K_s$ are depending on $m$, and so we also write $K_w(m) \equiv K_w$ and $K_s(m) \equiv K_s$.  
	
	\begin{remark}[Inequality between Kalton's constants]\label{remconst}
		Here, we draw a clear picture of inequalities between all Kalton's constants $K$, $K_w$ and $K_s$. Denote by $K(m)$ the optimal Kalton constant for 1-additive functions defined on $2^{\Omega_m}$, to emphasize $m$ dependence. \smallskip 
		Clearly, if $f$ is $\varepsilon$-additive, then it is weakly-$\varepsilon$-modular and the converse is not true as $f(\varnothing)$ may be non-zero. However if $f$ is weakly-$\varepsilon$-modular and $f(\varnothing) =a \neq 0$, then by shifting, we get $g=f - a\cdot \mathds{1}_{\mathcal{F}}$, which is $\varepsilon$-additive, as for any $A,B\in \mathcal{F}$ we have \begin{align*}| g(A) + g(B) - g(A \cup B) | = &
		| f(A)-a + f(B)-a - f(A \cup B)+a|\\
		= &
		| f(A) + f(B) - f(A \cup B)- f(A \cap B)|
		\\\leqslant & \varepsilon.
		\end{align*}
		Similarly, for any modular function $\nu$ with $\nu(\varnothing) = b \neq 0$ one can construct an additive function by setting $\mu = \nu -b\cdot \mathds{1}_{\mathcal{F}}$.
		
		Let $g$ be weakly-$\varepsilon$-modular so that $g(\varnothing)=a \neq 0$. Set $f= g -a\mathds{1}_{\mathcal{F}}$ and note that it is $\varepsilon$-additive. There is an additive function $\mu$ so that
		$$ \sup_{A \in \mathcal F}| f(A) - \mu(A)| \leqslant \varepsilon K(m).$$
		Let $\nu_a = \mu + a\mathds{1}_{\mathcal{F}}$ then  $\nu_a$ is modular and
		$\nu_a(\varnothing) =  a$ also 
		$$\sup_{A} | g(A) - \nu_a(A)| = \sup_{A}| f(A)+a - \mu(A)-a| = \sup_{A}|f(A) - \mu(A) | \leqslant K(m)\varepsilon.$$
		Thus, $K_w(m) \leqslant K(m)$. Now, let $f$ be $\varepsilon$-additive, so it is weakly-$\varepsilon$-modular. Consequently, there is a modular function $\nu$ so that 
		$\sup_{A \in \mathcal F} |f(A) - \nu(A)| \leqslant \varepsilon K_w(m)$. Assume that $\nu(\varnothing) = c \neq 0$. Then $\mu = \nu - c \mathds{1}_{\mathcal{F}}$
		is additive and
		$$|f(A) - \mu(A)| \leqslant |f(A) -\nu(A) + c | \leqslant |f(A) - \nu(A)| + |f(\varnothing) - c| \leqslant 2 K_w(m).$$
		Thus, $K(m) \leqslant 2 K_w(m)$. Hence, we have the following inequalities between Kalton's constants
		\begin{equation}\label{eqn:triple_k}\tfrac{1}{2}K_w(m) \leqslant K_s(m) \leqslant K_w(m) \leqslant K(m) \leqslant 2K_w(m).\end{equation}

	\end{remark}
	
	\subsection*{Lower bounds}
	
	The results concerning estimating  $K$, $K_w$ and $K_s$ from below have been so far rather scarce. In 1987, Pawlik  published a~paper \cite{Paw}, where Kalton's constant $K$ was estimated from below by 3/2.
	Recently, his result has been reviewed in \cite[Appendix A, Appendix C]{Feige}. Moreover, Feige \emph{et al.}~have proved that  $K_s \geqslant 1$ \cite[Theorem 1.2]{Feige}. 

	\section{Main results}
	The aim of the present paper is to improve known lower bounds on Kalton's constant by obtaining the following inequality.
	
	\begin{theorem}\label{main1}$K \geqslant 3$. 
	\end{theorem}

	In order to prove the Main Theorem, we require the following fact. Let $\mathcal{F}_m:=2^{\Omega_m}$, that is, the power set of an $m$-element set (so that $\mathcal{F}_m$ has $2^m$ elements) and denote by $K(m)$ the optimal Kalton constant for 1-additive functions defined on $\mathcal{F}_m$ only. Then the sequence $(K(m))_{m=1}^\infty$ is increasing and $$K = \lim_{m\to \infty} K(m) = \sup_{m\in \mathbb N} K(m).$$
	(This follows from a standard compactness argument; see the first paragraph of the proof of \cite[Theorem 4.1]{KRo} for details.) In other words, it is sufficient to work with finite set algebras.\smallskip
	\begin{remark}It is clear that both sequences $(K_w(m))_{m=1}^{\infty}$, $(K_s(m))_{m=1}^{\infty}$ are increasing.
		As $K$  is the limit of $K(m)$ as $m\to \infty$, we have
		\begin{equation}\label{eqn:triple_knon}\ K_w(m) \leqslant K(m) \leqslant \sup_{m \in \mathbb{N}} K(m) = K,\end{equation}
		then, by monotone convergence theorem,
		the sequence $(K(m))_{m \in \mathbb{N}}$ is convergent, and similarly $K_w := \lim_{m \to \infty}K_w(m) = \sup_{m \in \mathbb{N}}K_w(m)$ similarly, as
		$K_s(m) \leqslant K_w(m)$. Then, 
		$K_s := \lim_{m \to \infty}K_s(m)=\sup_{m \in \mathbb{N}}K_s(m)$. This leads to the estimates
		\begin{equation}\label{eqn:triple_knon}\tfrac{1}{2}K_w \leqslant K_s \leqslant K_w\leqslant K\leqslant 2K_w,\end{equation}
		which show that indeed there is no dependence on $m$. 
	\end{remark}
	As an immediate corollary to our Main Theorem and (\ref{eqn:triple_knon}), we obtain the lower bound for $K_w= \lim_{m \to \infty}K_w(m)$ (see also \cite[Theorem 1.3]{Feige}, where it is proved that already $K_w(20) \geqslant \tfrac{3}{2}$).
	\begin{corollary}
		$K_w \geqslant \frac{3}{2}$.
	\end{corollary}
	\section{Proof of the main result}
	
	Let $\Omega_{k,n}$ be a set of cardinality $n\cdot k$; we write $\Omega_{k,n}$ as the disjoint union of sets $X_1, \ldots, X_n$ each set having cardinality $k$, where $n, k \geqslant 2$. We define $f_{k,n}$ by setting 
	\begin{itemize}
		\item $f_{k, n}(\varnothing) = 0$;
		\item $f_{k,n}(A) = 3$ for every set $A$ with $A\cap X_j\neq \varnothing$ for all $j\leqslant n$ and
		$A\cap X_j = X_j$ for at least one $j$;
		\item $f_{k,n}(B) = 1$ for all other sets $B$.
	\end{itemize}
	In particular, $f_{k,n}(\Omega_{k,n}) = 3$. It is a matter of direct verification that each function $f_{k,n}$ is $1$-additive and so weakly-$1$-modular, which yields also that $f_{k,n}$ is $2$-modular.
	
	\smallskip
	
	\begin{proof}[Proof of the Main Theorem] Let $\mu_{k,n}$ be a measure that minimises the distance from $f_{k,n}$ to the space of measures on $\Omega_{k,n}$. Choose indices $i_1,\ldots, i_n$ that realise $\gamma^1_{k,n}, \ldots, \gamma^n_{k,n}$, where 
		$$ \gamma^j_{k,n} = \min_{i\in X_j} | \mu_{k,n}(\{i\})|\quad (j = 1, \ldots, n).$$
		We \emph{claim} that for all $j$ and $n$ we have $\gamma^j_{k,n} \to 0$ as $k\to \infty$. Assume not. Then $\gamma^j_{k,n} \geqslant \gamma$ for some $\gamma > 0$ and infinitely many $k$. 
		
		Let $$M = \sup_{k,n} \sup_{A\subseteq \Omega_{k,n}} |\mu_{k,n}(A)|.$$
		If $M = \infty$, the theorem would have been proved, so we may assume that $M$ is finite.  (Of course, it follows  from the Kalton--Roberts theorem that $M\leqslant 44.5+3$, but there is no need to invoke such a deep result here.) As $k$ increases over the chosen infinite set, the number of those $i\in X_j$ for which $\mu_{k,n}(\{i\})$ are either all positive or all negative increases to infinity; let $A_j$ denote the subset of $X_j$ comprising such elements of the same sign. In particular, $$|\mu_{k,n}(A_k)| \geqslant \gamma \cdot |A_k|\to \infty$$ as $k\to \infty$; a~con\-tr\-adiction. \smallskip
		
		Let us note that
		\begin{eqnarray*}n\cdot K(k\cdot n)& \geqslant& n \cdot \sup_{A\subseteq \Omega_{k, n}} |f_{k,n}(A) - \mu_{k,n}(A)|\\
			& \geqslant & \sum_{j=1}^n |f_{k,n}(X_j\cup \{i_\ell\colon \ell \neq j\}) - \mu_{k,n}(X_j) - \mu_{k,n}(\{i_\ell\colon \ell \neq j\})|\\
			& = & \sum_{j=1}^n |3 - \mu_{k,n}(X_j) - \mu_{k,n}(\{i_\ell\colon \ell \neq j\})|\\
			& \geqslant & \sum_{j=1}^n (3 - \mu_{k,n}(X_j) - \mu_{k,n}(\{i_\ell\colon \ell \neq j\}))\\
			& = & 3n - \mu_{k,n}(\Omega_{k,n})  - \sum_{j=1}^n \mu_{k,n}(\{i_\ell\colon \ell \neq j\}). \end{eqnarray*}
		We have  
		$$ K(k\cdot n) \geqslant 3 - \frac{1}{n}\mu_{k, n}(\Omega_{k,n}) - \frac{1}{n}\sum_{j=1}^n \mu_{k,n}(\{i_\ell\colon \ell \neq j\})),$$
		which shows that 
		\begin{eqnarray*} K &\geqslant& 3 - \frac{1}{n} \limsup_{k\to\infty} \mu_{k, n}(\Omega_{k,n}) - \frac{1}{n} \limsup_{k\to\infty} \sum_{j=1}^n \mu_{k,n}(\{i_\ell\colon \ell \neq j\}) \\
			&=& 3 - \frac{1}{n} \limsup_{k\to\infty} \mu_{k, n}(\Omega_{k,n}) - \frac{1}{n} \limsup_{k\to\infty} \sum_{j=1}^n \sum_{\ell \neq j} \gamma^\ell_{k,n} \\
			& = & 3 - \frac{1}{n} \limsup_{k\to\infty} \mu_{k, n}(\Omega_{k,n})\\
			& \geqslant & 3 - \frac{M}{n},
		\end{eqnarray*}
		because $$\sum_{j=1}^n \sum_{\ell \neq j} \gamma^\ell_{k,n}\to 0$$ as $k\to \infty$ (and $n$ is fixed).
		
	\end{proof}

	\section{The constants $K(m)$}\label{Kk}
	The proof of the Main Theorem has an asymptotic nature as it involves all constants $K(m)$ at once. Given the value of $m$, it would be thus desirable to find lower (and upper) estimates for $K(m)$ as well. This can be achieved by estimating the distances of the functions appearing in the proof of the Main Theorem to the space of measures.\smallskip
	
	We start with the following lemma, which asserts that it is always possible to find a~measure minimising the distance to $f_{k,n}$ that is constant on singletons from the respective partitions. (As the supremum norm that we consider here is not strictly convex, there is no guarantee for the uniqueness of the element that minimises a distance to a subspace.)\smallskip 
	
	For $n,k\in \mathbb N$, denote by $\mathcal{F}_{k,n}$ the power-set of $\Omega_{k,n}$ and let $\mathcal{S}_{k,n}$ be the set of all self-bijections of $\Omega_{k,n}$ that leave each set $X_j$ invariant ($j\leqslant n$). Then $\mathcal{S}_{k,n}$ has exactly $(k!)^n$ elements.
	
	\begin{lemma}\label{regular}Let $n,k\in \mathbb N$. Then there exists a measure $\nu$ that minimises the distance from $f_{k,n}$ to the space of measures on $\Omega_{k,n}$ with the property for every $j\leqslant n$ the function $x\mapsto \nu(\{x\})$ is constant on $X_j$.\end{lemma}
	\begin{proof}
		Let $\mu$ be any measure that minimises the distance from $f_{k,n}$ to the space of measures. For any self-bijection $\sigma$ of $\Omega_{k,n}$, the composition $\mu\circ \sigma$ defines a measure again. Let us observe that the measure
		$$\nu = \frac{1}{(k!)^n} \sum_{\sigma \in \mathcal{S}_{k,n}} \mu\circ \sigma$$
		has the desired properties. Indeed, it is clear that the function $x\mapsto \nu(\{x\})$ is constant on the respective sets $X_j$ ($j\leqslant n$) as we consider only bijections that leave each set $X_j$ invariant. Let then prove that $\nu$ also minimises the distance to the space of measures. Indeed, by convexity of balls (here, in $\ell_\infty(\mathcal{F}_{k,n})$), we have
		\begin{eqnarray*}\sup_{A\in \mathcal{F}_{k,n}}\left|f_{k,n}(A)-\frac{1}{\left(k!\right)^n}\sum_{\sigma \in \mathcal{S}}(\mu\circ \sigma)(A)\right| &\leqslant  \sup_{A\in \mathcal{F}_{k,n}}\left|f_{k,n}(A)-\mu(A)\right|.\end{eqnarray*}
		As $\mu$ was chosen to minimise the distance, the proof is complete.
	\end{proof}

	\subsection{The case $n=2$} Let $n=2$, $k\in \mathbb N$, and let $\nu$ be a measure as in the statement of Lemma~\ref{regular}. In that case, $\Omega_{k,2}=X_1\cup X_2$. Denote $x=\nu(X_1)$ and $y=\nu(X_2)$. In this case, we have essentially three types of sets to consider: 
	\begin{itemize}
		\item $X_1\cup\{\omega_2\}$, where $\omega_2\in X_2$,
		\item $X_2\cup\{\omega_1\}$, where $\omega_1\in X_1$,
		\item $\Omega_{k,2} \setminus \{\omega_1,\omega_2\}$, where $\omega_1\in X_1$ and $\omega_2\in X_2$.
	\end{itemize}
	Thus, we seek to minimise the following expressions simultaneously:
	$$\left|x + \frac{y}{k}-3\right|,\; \left|y + \frac{x}{k}-3\right|,\; \left|x+y - \frac{x+y}{k}-1\right|$$
	with respect to $(x,y)$. We then arrive at the following system of equations:
	$$\left\{ \begin{array}{ll}\frac{k-1}{k}(x+y)-1 = 3 - x - \frac{y}{k} \\
	\frac{k-1}{k}(x+y)-1 = 3 - y - \frac{x}{k},\end{array}\right.$$
	which has the unique solution:
	$$\left\{ \begin{array}{l}x = \frac{4k}{3k-1}\\
	y = \frac{4k}{3k-1}.\end{array}\right. $$
	
	In that case, the lower estimates for $K(m)$ are suboptimal as asymptotically they yield  the inequality $K\geqslant 5/3$.\bigskip
	
	\subsection{The case $n\geqslant 3$} Analogously to the case $n=2$, let us denote $x_i = \nu(X_i)$ for $i=1, \ldots, n$. For every $j\leqslant n$, let us pick $\omega_j \in X_j$ ($j\leqslant n$). By Lemma~\ref{regular}, we may restrict our attention to measures that assume equal values on singletons from the respective sets $X_j$. In other words, it is enough to consider the following sets:
	\begin{itemize}
		\item $\Omega_{k,n} \setminus \{\omega_\ell \colon \ell\leqslant n\}$;
		\item $X_j\cup\{\omega_\ell\colon \ell\leqslant n, \ell \neq j\}\quad (j\leqslant n)$.
		
	\end{itemize}
	Thus, this time, we seek to minimise the following expressions simultaneously:
	$$\left|\sum_{\ell \leqslant n} x_\ell - \frac{\sum_{\ell \leqslant n} x_\ell}{k}-1\right|,\; \left|x_j + \sum_{\ell\neq j} \frac{x_\ell}{k}-3\right|\; (j\leqslant n)$$
	with respect to $(x_1, x_2, \ldots, x_n)$. In particular, for $j\leqslant n$, we have \begin{equation}\label{eq1}\frac{k-1}{k}\sum_{\ell\leqslant n} x_\ell - 1 = 3 - x_j - \sum_{\ell\neq j} \frac{x_\ell}{k}.\end{equation}
	The sum $t = \sum_{\ell\leqslant n} x_\ell$ may be then computed by adding these equations together. More specifically, $t = \frac{4nk}{(n+1)k-1}$. It follows from \eqref{eq1} that for any $j\leqslant n$ we have
	$$\frac{k-1}{k} t - 1 = 3 - \frac{k-1}{k}x_j - \frac{t}{k}.$$
	Finally, for every $j\leqslant n$, we have
	$$x_j = \frac{k}{k-1}\Big(4 - \frac{4nk}{(n+1)k - 1}\Big) = \frac{4k}{(n+1)k-1}.$$
	Since the double sequence
	$$a_{k,n} = 3 - \frac{4k}{(n+1)k-1} - \frac{n-1}{k}\frac{4k}{(n+1)k-1}$$
	converges to $3$ as $k,n\to \infty$, we may estimate $K(m)$ from below by $a_{k,n}$, where $k,n$ are such that $m=k\cdot n$.\smallskip
	
	In particular, we could restrict, for example, to $n=k$. In this case, for $n=k=10$, we get approximately $2.305$, for $n=k=20$, it is approximately $2.628$, and for $n=k=200$, we arrive at $2.96$.
	
	\section{Closing remarks}
	Feige, Feldman, and Talgam-Cohen remarked that \emph{obtaining good lower bounds on $K_s$ is also not easy. Part of the difficulty is
		that even if one comes up with a function $f$ that is a~candidate to yield the lower bound,
		verifying that it is $\varepsilon$-modular involves checking roughly $2^{2^n}$ approximate modularity equations} \cite[p.~69]{Feige}.\smallskip
	
	Motivated by the above statement, we have found a suitable candidate for the function(s) $f_{k,2}$ using a~{Python} script, which gave us a lower estimate of $5/3$ for $K$. Subsequently, we added more degrees of freedom (by defining $f_{k,n}$) in an analogous manner. Let us briefly explain our approach, which would probably make the proof of the main result less \emph{ad hoc}.\smallskip
	
	We consider the set $\Omega_{k,2}$ for $k\in \mathbb N$ and $k \geqslant 2$ so that $|\Omega_{k,2}| \geqslant 4$. Let $A = [a_{ij}]_{i,j=1}^3$ be a~real matrix. We then define a function $f\colon \mathcal F_{k,2}\to \mathbb R$ by asserting that 
	
	$$\begin{array}{lll}
	f_k(\varnothing) = 0, &
	f_k(Y^{\prime}) = a_{12}, &  f_k(X_2)= a_{13},\\
	f_k(X^{\prime})= a_{21}, & f_k(X^{\prime} \cup Y^{\prime}) = a_{22}, &
	f_k(X^{\prime} \cup X_2) = a_{23},\\
	f_k(X_1) = a_{31}, & f_k(X_1 \cup Y^{\prime}) = a_{32}, & f_k(X_1 \cup X_2) = a_{33}
	\end{array}$$
	as long as
	and $X^{\prime}$, $Y^{\prime}$ are proper, non-empty subsets of $X_1$ and $X_2$, respectively. 
	
	\begin{lemma}\label{lem: oneadd}
		The function $f$ is 1-additive (weakly-1-modular) if and only if the following conditions are satisfied:
		\begin{romanenumerate}
			\item $|a|\leqslant 1$ for $a \in \{ a_{12}, a_{21}, a_{22}\}$.
			\item $|2a_{12} - a_{13}|\leqslant 1$,
			\item $|2a_{21} - a_{31}| \leqslant 1$, 
			\item $|a_{13}+a_{31}-a_{33}|\leqslant 1$, 
			\item $|a_{12}+a_{21}-a_{22}|\leqslant 1$, 
			\item $|2a_{22}-b| \leqslant 1$ for $b \in \{a_{23}, a_{32}, a_{33}\}$,
			\item $|a_{12}+a_{22}-a_{23}| \leqslant 1$,
			\item $|a_{21}+a_{22}-a_{32}| \leqslant 1$,
			\item $|  a_{33} - c|  \leqslant  1$ for $c \in \{ a_{32} + a_{12}, a_{23} + a_{21}\}$.
		\end{romanenumerate}
		In particular,  $a_{13}, a_{31}, a_{23}, a_{32}, a_{33} \in [-3, 3]$.
	\end{lemma}
	\begin{proof}
		Straightforward verification. 
	\end{proof}
	
	Effectively, Pawlik's construction corresponds to the matrix
	$$\left[\begin{smallmatrix} 0 & -1 & -3 \\ 1 & 0 & -1 \\ 3 & 1 & 0 \end{smallmatrix} \right].$$ Having implemented the conditions from Lemma~\ref{lem: oneadd} in Python, we run a simple script that listed for us \emph{all} 1-additive functions of that form that take values from the list $(-3, -2.5, -2, \ldots, 2, 2.5, 3)$. (By Lemma~\ref{lem: oneadd}, the numbers $-3$ and $3$ are extremal values for the range of such functions.) Overall, we found in total 38,034 such functions that are non-zero (excluding those that differ only by the sign, we had only 19,017 functions to investigate after all). Using a~convex optimisation solver SCS (\cite{ocpb:16,scs}), we filtered out those functions whose distance to the space of measures is at least 1.4 in the case $k=4$ (that is, functions on an 8-element set), having found only two:
	
	$$\left[ \begin{smallmatrix} 0 & 1 & 1 \\
	1 & 1 & 3 \\
	1 & 3 & 3\end{smallmatrix} \right], \left[ \begin{smallmatrix} 0 & 1 & 1 \\
	1 & 1 & 3 \\
	3 & 3 & 3\end{smallmatrix} \right].$$
	Obviously, the former one corresponds to functions $f_{k,2}$ that we consider in the present paper.
	
	\subsection*{Acknowledgements.} 
	We would like to thank the anonymous referee for her/his corrections and comments; in particular, for pointing out the slip concerning the relation between $K_w$ and $K$ in the submitted version; the respective  comments have been incorporated into Remark~\ref{remconst}.\smallskip
	
	Moreover, we are indebted to Tomasz Kochanek who has introduced to us the problem of improving the bounds for Kalton's constant. The first-named author's visit to Prague in December 2019, during which some part of the project
	was completed, was supported by funding received from GA\v{C}R project 19-07129Y; RVO
	67985840, which is acknowledged with thanks. The research was supported by the grant SONATA BIS no.~2017/26/E/ST1/00723.

	\bibliographystyle{plain}
	\bibliography{bibliography}
\end{document}